\documentclass{jonasart}

\DeclareMathOperator{\ch}{char}
\DeclareMathOperator{\id}{id}

\title{$\sigma$-matching and interchangeable structures on truncated polynomial algebras}

\authors{Kobiljon Abdurasulov, Jobir Adashev and Feruza Toshtemirova}

\authorinfo[K. Abdurasulov]{CMA-UBI, University of Beira Interior, Portugal; Romanovsky Institute of Mathematics, Uzbekistan}{abdurasulov0505@mail.ru}

\authorinfo[J. Adashev]{Romanovsky Institute of Mathematics, Uzbekistan; Chirchiq State Pedagogical University, Uzbekistan}{adashevjq@mail.ru}

\authorinfo[F. Toshtemirova]{Romanovsky Institute of Mathematics, Uzbekistan}{feruzaisrailova45@gmail.com}

\abstract{We describe $\sigma$-matching, interchangeable and, as a consequence, totally compatible products on truncated polynomial algebras.}



\acknowledgments{The work is supported by FCT 2023.08952.CEECIND and UID/00212/2025.}

\VOLUME{1}
\ISSUE{1}
\NUMBER{9}
\DOI{https://doi.org/10.46298/jonas.17682}
\editinfo{March 10, 2026}{August 25, 2026}{Maxime Fairon}

\begin{document}

	\section*{Introduction}

	The study of algebraic structures equipped with multiple compatible products has become a~central theme in non-associative algebras, with deep connections to mathematical physics, integrable systems, operad theory, and deformation theory. In particular, the investigation of algebras admitting two bilinear operations $\cdot$ and $\star$ -- satisfying specific compatibility conditions has led to the development of a~rich theory encompassing compatible associative algebras~\cite{Odesskii-Sokolov06, Dotsenko09}, compatible Lie algebras~\cite{Golubchik-Sokolov02, Golubchik-Sokolov05}, compatible pre-Lie algebras~\cite{Abdelwahab-Kaygorodov-Makhlouf24}, Rota–Baxter algebras~\cite{U, zgg20}, and dendriform structures. Among these, the concepts of $\sigma$-matching, interchangeable, and totally compatible products provide a~systematic algebraic framework for describing how two operations can coherently coexist on the same vector space.

	The notion of a~pair of compatible Poisson or Lie brackets seems first to have appeared in mathematical physics~\cite{Magri78,Reyman-Semenov89, Bolsinov91}. Odesskii and Sokolov investigated this structure extensively and obtained significant results~\cite{Odesskii-Sokolov06, Odesskii-Sokolov08}.

	At a~more abstract level, compatible operations have also been investigated within the framework of operad theory. In particular, Dotsenko and Khoroshkin computed the dimensions of the graded components of the operads associated with pairs of compatible Lie brackets and compatible Poisson brackets~\cite{Dotsenko-Khoroshkin07}. Algebras equipped with two matching associative products, known as matching dialgebras, together with their relations to other classes of algebras, were studied in~\cite{zbg12}. In addition, totally compatible associative products and the corresponding operads were examined in~\cite{zbg13,Zhang13}. Subsequently, the notions of compatibility, matching compatibility, and total compatibility were extended to families of specific algebras, including associative algebras~\cite{zgg20}. Later, these ideas were further generalized to families of algebras defined over arbitrary operads~\cite{zgg24}.

	In the study of compatible structures, compatibility can be evaluated in two ways. For varieties whose signature naturally contains multiple operations (such as Poisson algebras), one can consider compatibility across all operations simultaneously~\cite[Definition A]{Strohmayer08}. Conversely, for a~variety defined by a~single operation—such as the variety of associative algebras—one can introduce an additional product that is compatible with the original fixed operation. This latter approach has been investigated for general associative algebras~\cite{Khr}, the strictly upper triangular matrix algebra~\cite{Khr1}, and the Jacobson radical of incidence algebras~\cite{Khr2}. Moreover, some classes of compatible algebras, as well as low-dimensional compatible algebras, have been classified~\cite{aae24, Ladra-Cunha-Lopes24}.

	Let $V$ be a~vector space over a~field $K.$ Two associative products $\cdot$ and $\star$ on $V$ are called compatible if
	\begin{align*}
		(a\cdot b)\star c + (a\star b)\cdot c= a~\cdot(b\star c)+a\star (b\cdot c).
	\end{align*}
	Let us denote the terms on the left-hand side and the right-hand side by $(i), (ii), (iii)$ and $(iv),$ respectively. In general, compatible associative algebras are well studied. We list some of their subclasses. If $(i) = (iii)$ and $(ii) = (iv)$, then this subclass is called $\id$-matching. If $(i) = (iv)$ and $(ii) = (iii),$ then it is called $(12)$-matching. If all the expressions coincide, namely $(i) = (ii) = (iii) = (iv),$ then this subclass is called totally compatible.

	This article focuses on these compatibility structures within a~very concrete and well-understood class of algebras: truncated polynomial algebras over the complex field $\mathbb{C}$~\cite{MO}. These algebras are, in a~sense, the simplest non-trivial nilpotent associative algebras: they are characterized by having maximal nilpotency index relative to their dimension. Up to isomorphism, the $n$-dimensional truncated polynomial algebra $\mathbb{C}_n [x]=\mathbb{C}[x]\big/(x^{n+1})$ has a~canonical basis $\{e_1,\dots,e_n\}$ with multiplication $e_i \cdot e_j = e_{i+j}$ (when $i+j \leq n$) and all other products zero. In some articles, this algebra is named null-filiform associative algebra.

	In general, filiform algebras have been thoroughly investigated in numerous varieties of algebras; see, for example, filiform left-symmetric algebras~\cite{DEKIMPE}, Rota-type operators on null-filiform associative algebras~\cite{KKL}, and central extensions of null-filiform associative algebras~\cite{KLP}. In addition, two-operation algebras whose one operation is isomorphic to a~filiform algebra have recently attracted considerable attention; see, for instance, in the study of transposed $\delta$-Poisson algebras~\cite{ABF,DEF}.

	The fact that $\id$-matching, interchangeable, and totally compatible structures coincide on $(\mathbb{C}_n [x], \cdot)$ is a~distinctive feature of this class and underscores the strong constraints imposed by the underlying algebra’s structure.

	The paper is organized as follows. Section~\ref{prelim} provides the necessary definitions and fundamental results concerning associative algebras. Section~\ref{com_a lgebra} presents general results concerning the characterization of $\sigma$-matching and interchangeable structures in associative commutative algebras. Section~\ref{id-inter-total} classifies interchangeable (hence $\id$-matching and totally compatible) structures
	(see Theorem~\ref{int}). Section~\ref{12-matching-section} classifies $(12)$-matching structures (see Theorem~\ref{12-matching-algebras}).

	\section{Preliminaries}\label{prelim}

	\begin{definition}
		Two associative products $\cdot$ and $\star$ on $V$ are called compatible if
		\begin{align*}
			(a\cdot b)\star c + (a\star b)\cdot c= a~\cdot(b\star c)+a\star (b\cdot c).
		\end{align*}
	\end{definition}

	\begin{definition}
		Adopting the terminology of \textrm{\cite{zgg23}}, we say that two associative products $\cdot$ and $\star$ on $V$ are
		\begin{enumerate}
			\item \textit{$id$-matching}, if
			\begin{align}\label{id-ass-identity}
				(a\cdot b)\star c = a\cdot(b\star c) \text{ and } (a\star b)\cdot c = a\star(b\cdot c),
			\end{align}
			\item \textit{$(12)$-matching}, if
			\begin{align}\label{matching-ass-identity}
				(a\cdot b)\star c = a\star(b\cdot c) \text{ and } (a\star b)\cdot c = a\cdot(b\star c),
			\end{align}
			\item \textit{totally compatible}, if
			\begin{align}
				(a\cdot b)\star c = (a\star b)\cdot c = a\cdot (b\star c) = a\star (b\cdot c),
			\end{align}
		\end{enumerate}
		for all $a,b,c\in V$.
	\end{definition}

	\begin{definition}
		We say that two (not necessarily associative) products $\cdot$ and $\star$ on $V$ are \textit{interchangeable} if
		\begin{align}\label{interchangeable-ass-identity}
			(a\cdot b)\star c = (a\star b)\cdot c \text{ and } a\cdot(b\star c) = a\star (b\cdot c).
		\end{align}
	\end{definition}

	Note that any totally compatible products are interchangeable and $\sigma$-matching for every $\sigma \in S_2$. However, the converse does not hold in general
	(see \cite[Examples 1.4 and 1.5]{Khr}).

	The following examples show that the classes of $\id$-matching, $(12)$-matching and interchangeable structures on $(A,\cdot)$ are different and no class is in general contained in another one (although all of them contain the totally compatible structures).

	\begin{example}
		Let $A$ be the $4$-dimensional nilpotent associative algebra with multiplication table $e_1\cdot e_2 =e_4$. The product $*$ on $A$ given by $e_1 *e_1 =e_1$ and $e_1 *e_4 =e_4$ is clearly associative. Then,
		\begin{enumerate}
			\item $\cdot$ and $*$ are $id$-matching.
			\item $\cdot$ and $*$ are not $(12)$-matching:
			$e_4 =e_1\cdot e_2 =(e_1 * e_1 ) \cdot e_2 \neq e_1\cdot (e_1 * e_2 )=0. $
			\item $\cdot$ and $*$ are not interchangeable:
			$e_4 =e_1\cdot e_2 =(e_1 * e_1 )\cdot e_2 \neq (e_1\cdot e_1 ) * e_2 =0.$
		\end{enumerate}
	\end{example}

	\begin{example}
		Let $A$ be the $4$-dimensional nilpotent associative algebra with multiplication table $e_1\star e_3 =e_4$. The product $\circ$ on $A$ given by $e_1 \circ e_2 =e_3,$ $e_1\circ e_3 =e_4$. Then,
		\begin{enumerate}
			\item $\star$ and $\circ$ are $(12)$-matching.
			\item $\star$ and $\circ$ are not $id$-matching:
			$0=(e_1 \star e_1 ) \circ e_2 \neq e_1 \star (e_1 \circ e_2 )=e_1 \star e_3 = e_4. $
			\item $\star$ and $\circ$ are not interchangeable:
			$e_4 = e_1 \star e_3 = e_1 \star (e_1 \circ e_2 ) \neq e_1 \circ (e_1 \star e_2 )=0.$
		\end{enumerate}
	\end{example}

	\begin{example}
		Let $A$ be the 6-dimensional associative algebra with multiplication table $e_1\cdot e_5 =e_2,$ $e_5\cdot e_6 =e_3,$ $e_1\cdot e_3 =e_4,$ $e_2\cdot e_6 =e_4.$ The product $\ast$ on $A$ given by $e_1 \ast e_5 = e_3,$ $e_1 \ast e_2 = e_4.$
		\begin{enumerate}
			\item $\cdot$ and $\ast$ are interchangeable.
			\item $\cdot$ and $\ast$ are not $id$-matching:
			$0= (e_1\ast e_1 )\cdot e_5 \neq e_1\ast (e_1\cdot e_5 )=e_1\ast e_2 =e_4.$
			\item $\cdot$ and $\ast$ are not $(12)$-matching:
			$0=(e_1\cdot e_1 )\ast e_5 \neq e_1\ast (e_1\cdot e_5 )=e_1\ast e_2 =e_4.$
		\end{enumerate}
	\end{example}

	\begin{remark}
		For two associative products $\cdot$ and $\star$ on $V$ the following are equivalent:
		\begin{enumerate}
			\item[$\textrm{(1)}$] $\cdot$ and $\star$ are totally compatible;
			\item[$\textrm{(2)}$] $\cdot$ and $\star$ are interchangeable and at least one of the two equalities~\ref{id-ass-identity} and~\ref{matching-ass-identity} holds for some $\sigma\in S_2$;
			\item[$\textrm{(3)}$] $\cdot$ and $\star$ are $\sigma$-matching for some $\sigma\in S_2$ and at least one of the two equalities~\ref{interchangeable-ass-identity} holds;
			\item[$\textrm{(4)}$] $\cdot$ and $\star$ are $\sigma_1$-matching for some $\sigma_1\in S_2$ and at least one of the two equalities~\ref{id-ass-identity} and~\ref{matching-ass-identity} holds for $\sigma\ne\sigma_1$;
		\end{enumerate}
		If one works over a~base field K such that $\ch(K)\ne 2$, then each of the conditions $(1)-(4)$ is also equivalent to
		\begin{enumerate}
			\setcounter{enumi}{4}
			\item[$\textrm{(5)}$] $\cdot$ and $\star$ are interchangeable and compatible.
		\end{enumerate}
	\end{remark}

	For any associative algebra $(A,\cdot)$ and fixed $x\in A$, one defines the \textit{mutation of $\cdot$ by $x$} (see~\cite{ElduqueMyung}) to be the product $\cdot_x$ on $A$ given by $a\cdot_x b=a\cdot x\cdot b$ for all $a,b\in A$. The following should be well-known (cf.~\cite[Formula (9)]{Carinena-Grabowski-Marmo2000} and~\cite[Example 2.2]{Odesskii-Sokolov06}), but we couldn't find an explicit proof.

	\begin{lemma}[see~\cite{Khr}]

		Let $(A,\cdot)$ be an associative algebra. For any $x\in A$ the product $\cdot_x$ is associative and $\id$-matching with $\cdot$.
	\end{lemma}

	\begin{proposition}[see~\cite{Khr}]\label{id=mutation}
		Let $(A,\cdot)$ be a~unital associative algebra. Then the $\id$-matching structures on $A$ are exactly the mutations of $\cdot$.
	\end{proposition}

	\begin{definition}
		The polynomial algebra $\mathbb{C}[x]$ is defined as
		\[
			\mathbb{C}[x] = \left\{ \sum_{i \in \mathbb{N} } a_i x^i \;\middle|\; \quad a_i \in \mathbb{C} \right\},
		\]
		where $x$ is an indeterminate.
	\end{definition}

	\begin{definition}
		A truncated polynomial algebra over $\mathbb{C}$ is defined as the quotient algebra
		\[
			\mathbb{C}_n [x]=\mathbb{C}[x]\big /(x^{n+1}), \quad n \in \mathbb{N},
		\]
		where $(x^{n+1})$ is the ideal generated by $x^{n+1}.$
	\end{definition}

	\section{\texorpdfstring{$\sigma$-matching and interchangeable structures on commutative algebras}{sigma-matching and interchangeable structures on commutative algebras}}\label{com_a lgebra}

	Throughout the paper, all vector spaces and algebras are finite-dimensional and over the complex field $\mathbb{C}$ unless otherwise stated.

	\begin{lemma}\label{id=interchangeable=total}
		Let $\cdot$ and $\star$ be associative products that are $\id$-matching. If one of them is commutative, then they are totally compatible and, consequently, $(12)$-matching and interchangeable.
	\end{lemma}

	\begin{proof}
		Let associative products $\cdot$ and $\star$ be $\id$-matching, i.e.,
		\[
			(a\cdot b)\star c=a\cdot (b\star c), \quad (a\star b)\cdot c=a\star (b\cdot c).
		\]
		Assume that one of the products is commutative. Without loss of generality, let $\cdot$ be commutative. Then
		\[
			(a\cdot b)\star c = (b\cdot a)\star c
			=b\cdot (a\star c)=(a\star c) \cdot b
			= a\star (c\cdot b)
			=a\star (b\cdot c).
		\]
		Combining this with the $\id$-matching identities, we obtain
		\[
			(a\cdot b)\star c = a\cdot (b\star c) = (a\star b)\cdot c = a\star (b\cdot c).
		\]
		Hence, $\cdot$ and $\star$ are totally compatible. The remaining properties follow immediately.
	\end{proof}

	\begin{lemma}\label{Interchangeable_m utation} Let $\cdot$ and $\star$ be associative products on a~unital algebra that are interchangeable structures. If $\cdot$ is the commutative product of the unital algebra, then $\cdot$ and $\star$ are $\id$-matching.
	\end{lemma}

	\begin{proof}
		Let associative products $\cdot$ and $\star$ be interchangeable structures on the unital algebra $A$ such that condition~\eqref{interchangeable-ass-identity} holds.

		Since the $\cdot$ is the commutative product of the unital algebra. For arbitrary $a,b \in A$, we use the unit property and interchangeability condition:
		\begin{center}
			$a\star b=a \star (b\cdot e)= a~\cdot (b\star e)= a~\cdot (b\star (e \cdot e))=a \cdot (b \cdot (e \star e))= a~\cdot (e \star e) \cdot b.$
		\end{center}
		Denoting $x:= e \star e$, we obtain
		\begin{center}
			$a \star b = a~\cdot x \cdot b = a~\cdot_x b,$
		\end{center}
		which shows that the product $\star$ is a~$\cdot_x$ mutation. By Proposition~\ref{id=mutation}, any such mutation yields $\id$-matching products. Therefore, $\cdot$ and $\star$ are $\id$-matching.
	\end{proof}

	\begin{lemma}\label{12_m utation}
		Let $\cdot$ and $\star$ be associative products on a~unital algebra that are $(12)$-matching. If $\cdot$ is the commutative product of the unital algebra, then $\cdot$ and $\star$ are $\id$-matching.
	\end{lemma}

	\begin{proof}
		Let associative products $\cdot$ and $\star$ be $(12)$-matching on the unital algebra $A,$ such that condition~\eqref{matching-ass-identity} is satisfied.

		Since the $\cdot$ is the commutative product of the unital algebra. For arbitrary $a,b \in A$, using the unit property and the $(12)$-matching identity, we compute
		\begin{center}
			$a\star b=a \star (b\cdot e)= (a \cdot b) \star e = ((a \cdot b) \star e) \cdot e = (a \cdot b) \cdot (e \star e)= a~\cdot (e \star e) \cdot b.$
		\end{center}
		Denoting $x:= e \star e$, we obtain
		\begin{center}
			$a \star b = a~\cdot x \cdot b = a~\cdot_x b,$
		\end{center}
		which shows that the product $\star$ is a~$\cdot_x$ mutation. By Proposition~\ref{id=mutation}, any such mutation yields $\id$-matching products. Therefore, $\cdot$ and $\star$ are $\id$-matching.
	\end{proof}

	\begin{remark}
		If $\cdot$ and $\star$ are $\id$-matching on commutative algebra, then they are totally compatible, $(12)$-matching, and interchangeable. The converse implication holds only in the class of unital commutative algebras.
	\end{remark}

	\begin{example}
		Let $A$ be the 3-dimensional associative commutative algebra $e_1\cdot e_1 =e_2,$ $ e_1\cdot e_2 =e_3,$ $e_2\cdot e_1 =e_3.$ The associative product $\star$ on $A$ given by $e_1 \star e_1 = e_1,$ $ e_1 \star e_2 = e_2 + e_3, $ $e_2 \star e_1 = e_2 + e_3.$
		\begin{enumerate}
			\item $\cdot$ and $\star$ are $(12)$-matching.
			\item $\cdot$ and $\star$ are not $id$-matching:
			$e_3 = e_1\cdot e_2 =(e_1\star e_1 )\cdot e_2 \neq e_1\star (e_1\cdot e_2 )=e_1\star e_3 =0.$
		\end{enumerate}
	\end{example}

	\begin{corollary}
		Let $A$ be a~unital associative and commutative algebra and $(A, \cdot, \star)$ be a~$\sigma$-matching and interchangeable structure on $A$. The following statements are equivalent:
		\begin{enumerate}
			\item $\cdot$ and $\star$ are $\id$-matching;
			\item $\cdot$ and $\star$ are $(12)$-matching;
			\item $\cdot$ and $\star$ are totally compatible;
			\item $\cdot$ and $\star$ are interchangeable.
		\end{enumerate}
	\end{corollary}

	\begin{proof}
		This is a~consequence of Lemmas~\ref{Interchangeable_m utation} and~\ref{12_m utation}.
	\end{proof}

	\section{\texorpdfstring{\(id\)-matching, interchangeable, totally compatible structures on \(\mathbb{C}_n [x]\)}{id-matching, interchangeable, totally compatible structures on (\mathbb{C}\_n [x])}}\label{id-inter-total}

	Now, we consider the truncated polynomial algebra $\mathbb{C}_n [x].$
	Recall that it has the basis $\{x, x^2, \dots, x^{n}\}$. If we denote the basis element by $e_i =x^i $ for $1 \le i \le n$. Then multiplication in $\mathbb{C}_n [x]$ is given by
	\begin{equation}\label{mu_0}
		e_i \cdot e_j =
		\begin{cases}
			e_{i+j}, & \text{if } 2 \le i+j \le n, \\
			0, & \text{if } i+j \ge n.
		\end{cases}
	\end{equation}

	\begin{theorem}\label{3.1}
		%Let $\mathbb{C}_n [x]$ be a~truncated polynomial algebra. 
        A linear map $\varphi: \mathbb{C}_n [x] \to \mathbb{C}_n [x]$ is an automorphism if and only if it has the form
		\begin{equation}\label{Aut}
			\varphi(e_1 )=\sum_{i=1}^n A_i e_i,
			\qquad
			\varphi(e_i ) = \sum_{j=i}^n \left( \sum_{\substack{k_1 +\cdots+k_i =j}} A_{k_1} A_{k_2} \cdots A_{k_i} \right) e_j, \quad 2 \le i \le n,
		\end{equation}
		where $A_1 \ne 0.$
	\end{theorem}

	\begin{proof}
		Assume that $\varphi: \mathbb{C}_n [x] \to \mathbb{C}_n [x]$ is an automorphism.

		First, observe that $\mathbb{C}_n [x]$ is generated by $e_1$. Indeed, by induction,
		$e_i = e_1^i, \quad 1 \le i \le n.$
		Since $\varphi$ is linear, we can write
		$\varphi(e_1 ) = \sum_{i=1}^n A_i e_i.$

		We claim that $A_1 \ne 0$. Suppose $A_1 = 0$. Then $\varphi(e_1 ) \in \langle e_2, \dots, e_n \rangle = (\mathbb{C}_n [x])^2$, which implies
		$\varphi(\mathbb{C}_n [x]) \subseteq (\mathbb{C}_n [x])^2.$
		Hence, $\varphi$ is not surjective, contradicting the assumption that $\varphi$ is an automorphism. Therefore, $A_1 \ne 0$.

		Since $\varphi$ is an algebra homomorphism, we have
		$\varphi(e_i ) = \varphi(e_1^i ) = \varphi(e_1 )^i.$
		Now compute
		\[
			\varphi(e_1 )^i = \left( \sum_{k=1}^n A_k e_k \right)^i.
		\]
		Using associativity and the multiplication rule, we obtain
		\[
			\varphi(e_1 )^i =
			\sum_{k_1,\dots,k_i}
			A_{k_1} A_{k_2} \cdots A_{k_i} \, e_{k_1 +\cdots+k_i}.
		\]
		Grouping terms by $j = k_1 + \cdots + k_i$, we get
		\[
			\varphi(e_i )
			=
			\sum_{j=i}^n
			\left(
			\sum_{\substack{k_1 +\cdots+k_i =j}}
			A_{k_1} \cdots A_{k_i}
			\right)
			e_j.
		\]
		On the other hand, assume that $\varphi$ is defined by
		\[
			\varphi(e_1 ) = \sum_{i=1}^n A_i e_i, \quad A_1 \ne 0, \qquad \text{and} \qquad
			\varphi(e_i ) = \sum_{j=i}^n \left( \sum_{k_1 +\cdots+k_i =j} A_{k_1} \cdots A_{k_i} \right) e_j.
		\]
		We show that $\varphi$ is an automorphism.

		First, we verify that $\varphi$ preserves multiplication. Since $e_i = e_1^i$, we have
		$\varphi(e_i ) = \varphi(e_1 )^i.$
		Then for any $i,j,$ where $2 \le i+j \le n:$
		\begin{center}
			$\varphi(e_i\cdot e_j )
			=
			\varphi(e_{i+j})
			=
			\varphi(e_1 )^{i+j}.$
		\end{center}
		On the other hand,
		\begin{center}
			$\varphi(e_i )\cdot \varphi(e_j )
			=
			\varphi(e_1 )^i \cdot \varphi(e_1 )^j
			=
			\varphi(e_1 )^{i+j}.$
		\end{center}
		Thus,
		$\varphi(e_i\cdot e_j ) = \varphi(e_i )\cdot \varphi(e_j ),$
		so $\varphi$ is an algebra homomorphism.

		Next, we prove that $\varphi$ is invertible. The matrix of $\varphi$ with respect to the basis $\{e_1, \dots, e_n\}$ is upper triangular with diagonal entries
		\[
			A_1, A_1^2, \dots, A_1^n.
		\]
		Since $A_1 \ne 0$, all diagonal entries are nonzero, hence the matrix is invertible. Therefore, $\varphi$ is bijective. Thus, $\varphi$ is an automorphism. The proof is complete.

		Note that, the theorem follows from the fact that $\mathbb{C}_n [x]$ is monogenic, i.e., generated by a~single element $e_1$. Hence, every automorphism is uniquely determined by the image of this generator.
	\end{proof}

	In the following theorem, we show that any interchangeable structure defined on the truncated polynomial algebra is $\id$-matching.

	\begin{theorem}\label{Interchangeable} Let $ \star$ be an interchangeable structure on the algebra $(\mathbb{C}_n [x],\cdot).$ Then the multiplication $\star$ is associative, and the products $\cdot$ and $\star$ are $\id$-matching.
	\end{theorem}

	\begin{proof}
		We first determine the multiplication $\star$ on the algebra $(\mathbb{C}_n [x],\cdot).$
		Set
		\begin{center}
			$e_1\star e_1 =\sum_{t=1}^{n}\alpha_{t}e_t.$
		\end{center}
		Since $\cdot$ and $\star$ are interchangeable, using the multiplication~\eqref{mu_0}, we obtain
		\begin{center}
			$e_i\star e_1 =(e_{i-1} \cdot e_1 )\star e_1 = (e_{i-1} \star e_1 )\cdot e_1;$\\
			$e_1\star e_i = e_1\star (e_1\cdot e_{i-1})=e_1\cdot (e_1\star e_{i-1}).$
		\end{center}
		Hence, by induction,
		\begin{center}
			$e_{i}\star e_1 =e_1\star e_{i}=\sum_{t=i}^{n}\alpha_{t-i+1}e_t, \quad 2\leq i \leq n-1.$
		\end{center}
		Next, using
		$e_i\star e_j =(e_{i-1}\cdot e_1 )\star e_j =(e_{i-1}\star e_1 )\cdot e_j,$ we obtain
		\begin{equation}\label{int6}
			e_i \star e_j =\sum_{t=i+j-1}^{n}\alpha_{t-i-j+2}e_t, \quad 2\leq i+j \leq n+1.
		\end{equation}
		We now prove that $\star$ is associative.
		\begin{align*}
			(e_i \star e_j ) \star e_k &= \sum_{t=i+j-1}^{n}\alpha_{t-i-j+2}e_t \star e_k =\sum_{t=i+j-1}^{n}\alpha_{t-i-j+2} \big( \sum_{s=t+k-1}^{n}\alpha_{s-t-k+2} e_s \big)\\
			&=\sum_{s=i+j+k-2}^{n} \left(\sum_{r=1}^{s-i-j-k+3}\alpha_{r} \alpha_{s-r-i-j-k+4}\right)e_s;\\
			e_i \star (e_j \star e_k ) &= \sum_{m=k+j-1}^{n}\alpha_{m-k-j+2}e_i \star e_m =\sum_{m=k+j-1}^{n}\alpha_{m-k-j+2} \big( \sum_{s=m+i-1}^{n}\alpha_{s-m-i+2} e_s \big)\\
			&=\sum_{s=i+j+k-2}^{n} \left(\sum_{l=1}^{s-i-j-k+3}\alpha_{l} \alpha_{s-l-i-j-k+4}\right)e_s.
		\end{align*}
		Hence, $(e_i \star e_j ) \star e_k = e_i \star (e_j \star e_k )$. Finally, using the explicit formula~\eqref{int6} for $e_i\star e_j,$ we compute
		\begin{center}
			$(e_i\star e_j )\cdot e_k = \sum_{t=i+j-1}^{n}\alpha_{t-i-j+2}(e_t\cdot e_k )=\sum_{t=i+j+k-1}^{n}\alpha_{t-i-j-k+2}e_t =e_i\star e_{j+k}=e_i\star (e_j\cdot e_k ).$
		\end{center}
		Thus, the products $\cdot$ and $\star$ are $\id$-matching.
	\end{proof}

	\begin{corollary}
		Let $\star$ be a~bilinear multiplication on $(\mathbb{C}_n [x],\cdot)$. Then the following statements are equivalent:
		\begin{enumerate}
			\item $\cdot$ and $\star$ are $\id$-matching;
			\item $\cdot$ and $\star$ are totally compatible;
			\item $\cdot$ and $\star$ are interchangeable.
		\end{enumerate}
	\end{corollary}

	\begin{proof}
		This is a~consequence of Lemma~\ref{id=interchangeable=total} and Theorem~\ref{Interchangeable}.
	\end{proof}

	Therefore, we can conclude that if $\cdot$ and $\star$ are $\id$-matching and algebra with the product $\cdot$ is truncated polynomial algebra, then the product $\star$ is determined by~\eqref{int6}. We denote by $Id(\alpha_1,\dots,\alpha_n )$ the class of algebras with such multiplications, that is,
	\begin{equation}\label{int7}
		Id(\alpha_1,\dots,\alpha_n ):
		\left\{
		\begin{array}{ll}
			e_i \cdot e_j =e_{i+j}, & 2\leq i+j \leq n, \\[1mm]
			e_i \star e_j =\sum_{t=i+j-1}^{n}\alpha_{t-i-j+2}e_t, & 2\leq i+j \leq n+1.
		\end{array}
		\right.
	\end{equation}
	In the following lemma, we provide a~criterion for determining when two algebras from the class~\eqref{int7} are isomorphic.

	\begin{lemma}\label{Lemma10}
		Two algebras $\mathfrak{L}^{\alpha_1,\dots,\alpha_n}$ and $\mathfrak{L}'^{\alpha_1 ',\dots,\alpha_n '}$ from the class~\eqref{int7} are isomorphic if and only if, there exist $A_1 \in \mathbb{C}^*,$ $A_2, \dots, A_n \in \mathbb{C},$ such that the following relation holds:
		\begin{equation}\label{inteq1}
			\sum_{i=1}^t\sum_{k_1 +\cdots+k_i =t} A_{k_1} \dots A_{k_i}\alpha_i ' = \sum_{i=1}^{t} \sum_{j=1}^{t-i+1} A_i A_j \alpha_{t-i-j+2}, \ \ 1\leq t\leq n.
		\end{equation}
	\end{lemma}

	\begin{proof}
		$\mathfrak{L}' \cong \mathfrak{L}$ if and only if, there exist an automorphism $\varphi$ of the algebra $(\mathbb{C}_n [x],\cdot),$ such that $\varphi(e_i \star' e_j ) = \varphi(e_i )\star \varphi(e_j ).$
		Since $\varphi$ is an automorphism of the algebra $(\mathbb{C}_n [x],\cdot),$ it follows from Theorem~\ref{3.1} that $\varphi$ has the form given in~\eqref{Aut}.

		First, we compute
		\begin{align*}
			\varphi(e_1 \star' e_1 ) 
            &= \varphi\Big(\sum_{i=1}^{n}\alpha_{i}'e_i\Big)
            =\sum_{i=1}^{n} \alpha_i '\sum_{t=i}^{n}\sum_{k_1 +\dots +k_i =t} A_{k_1}\dots A_{k_i}e_t \\
			&= \sum_{t=1}^{n} \sum_{i=1}^{t} \sum_{k_1 +\dots +k_i =t}\alpha_i ' A_{k_1} \dots A_{k_i}e_t.
		\end{align*}
		On the other hand,
        \begin{align*}
            \varphi(e_1) \star \varphi(e_1 )
            &= \sum_{i=1}^n A_i e_i \star \sum_{j=1}^n A_j e_j \\
            &= \sum_{i=1}^n \sum_{j=1}^n A_i A_j \left( \sum_{t=i+j-1}^{n} \alpha_{t-i-j+2} e_t \right) \\
            &= \sum_{i=1}^n \sum_{j=1}^n \sum_{t=i+j-1}^{n} A_i A_j \alpha_{t-i-j+2} e_t \\
            &= \sum_{t=1}^{n}\sum_{i=1}^{t} \sum_{j=1}^{t-i+1} A_i A_j \alpha_{t-i-j+2} e_t.
        \end{align*}
		By comparing the coefficients of the basis elements $e_t$, we obtain the system of relations~\eqref{inteq1}.

		On the other hand, the action of $\varphi$ on the higher-order basis elements $e_i$ ($i \geq 2$) is completely predetermined by Theorem~\ref{3.1}, and the parameters $A_i$ are already fixed by~\eqref{inteq1}. Since the multiplication tables for $\star$ and $\star'$ in~\eqref{int7} are structurally dependent on the image of the generator $e_1$, the equations arising from $\varphi(e_i \star' e_j ) = \varphi(e_i ) \star \varphi(e_j )$ for $i+j > 2$ identically reduce to the form $0=0$. Consequently, the map $\varphi$ preserves the second product for all basis elements, meaning $\mathfrak{L}' \cong \mathfrak{L}$. This completes the proof.
	\end{proof}

	Let $\alpha_s$ be the first non-vanishing parameter in the equality~\eqref{inteq1}, that is, $\alpha_s\neq 0$ and $\alpha_{i}=0$ for $1 \leq i \leq s-1.$ Then it follows immediately that
	\begin{equation}\label{alpha_s}
		\alpha_{s}'=\alpha_{s}A_1^{2-s}, \qquad \alpha_{i}'=0, \quad 1 \leq i \leq s-1.
	\end{equation}

	\begin{lemma}\label{intlem1} Let $\mathfrak{L}^{\alpha_1,\dots,\alpha_n}$ be an algebra from the class $Id(\alpha_1,\dots,\alpha_n )$. If $\alpha_1\neq0,$ then it is isomorphic to the algebra
		\[
			\mathbb{B}_1:
			\left\{
			\begin{array}{ll}
				e_i \cdot e_j =e_{i+j}, & 2\leq i+j \leq n, \\[1mm]
				e_i \star e_j =e_{i+j-1}, & 2\leq i+j \leq n+1.
			\end{array}
			\right.
		\]
	\end{lemma}

	\begin{proof}
		Since $\alpha_1\neq0,$ by choosing $A_1 = \frac{1}{\alpha_1,}$ $A_2 = -\frac{\alpha_2}{\alpha_1}$ in~\eqref{inteq1}, we obtain $\alpha_1 '=1$ and $\alpha_2 '=0.$ Hence, without loss of generality, we may assume that $\alpha_1 =1,$ $\alpha_2 =0.$

		We now prove by induction that one may further assume $\alpha_i = 0$ for all $2 \le i \le n$.

		The base case $i=2$ has already been established. Consider $i=3$. Assume $\alpha_1 = 1$ and $\alpha_2 = 0$. Choosing $A_1 = 1$ and $A_2 = 0$ in~\eqref{inteq1}, we obtain
		\begin{center}
			$\alpha_1 '=1, \quad \alpha_2 '=0, \quad \alpha_3 '=\alpha_3 + A_3.$
		\end{center}
		By taking $A_3 =-\alpha_3,$ we obtain $\alpha_3 '=0.$

		Assume now that for some $k \ge 3,$ we have $\alpha_1 =1,$ $\alpha_2 = \dots= \alpha_{k-1}=0,$ then choosing $A_1 = 1,$ $A_2 =\dots=A_{k-1}=0$
		in~\eqref{inteq1}, we have that
		\begin{center}
			$\alpha_1 '=1, \quad \alpha_i '=0, \ 2 \leq i \leq k-1, \quad \alpha_k '=\alpha_k + A_k.$
		\end{center}
		Choosing $A_k =-\alpha_k,$ we obtain $\alpha_k '=0.$
		Thus, by induction, any algebra from the class $Id(\alpha_1,\dots,\alpha_n )$ with $\alpha_1 \neq 0$ is isomorphic to an algebra with parameters $\alpha_1 ' = 1$ and $\alpha_i ' = 0$ for all $2 \le i \le n$, that is, to $\mathbb{B}_1$.
	\end{proof}

	Now, we consider the case when $\alpha_1 =0$ and $\alpha_s \neq 0$ for some $ 3 \le s \le n.$

	\begin{lemma}\label{intlem2} Let $\mathfrak{L}^{\alpha_1,\dots,\alpha_n}$ be an algebra from the class $Id(\alpha_1,\dots,\alpha_n )$. If $\alpha_1 =0$ and $\alpha_s \neq 0$ for some $3 \le s \le n$ then $\mathfrak{L}^{\alpha_1,\dots,\alpha_n}$ is isomorphic to the algebra
		\[
			\mathbb{B}_s (\alpha):\left\{
			\begin{array}{ll}
				e_i \cdot e_j =e_{i+j}, & 2\leq i+j \leq n, \\[1mm]
				e_i \star e_j = \alpha e_{i+j} + e_{s+i+j-2}, & 2\leq i+j \leq n-s+2, \\[1mm]
				e_i \star e_j = \alpha e_{i+j}, & n-s+3\leq i+j \leq n.\\[1mm]
			\end{array}
			\right.
		\]
	\end{lemma}

	\begin{proof}
		Let $\alpha_s$ be the first non-vanishing parameter, that is, $\alpha_{i}=0$ for $3 \leq i \leq s-1$ and $\alpha_s\neq 0$. Then, from~\eqref{inteq1}, it follows that
		\begin{gather*}
			\alpha_{1}'=0, \quad
            \alpha_2 '=\alpha_2, \quad
            \alpha_{i}'=0, \ 3 \leq i \leq s-1, \\
            A_1^{s-2}\alpha_{s}'=\alpha_{s}, \quad A_1^{s}\alpha_{s+1}'=A_1\alpha_{s+1} - (s-2) A_2\alpha_s.
		\end{gather*}
		We now prove by induction that one may further assume $\alpha_i =0$ for all $s+1 \leq i \leq n.$

		First, by choosing $A_1 = \sqrt[s-2]{\alpha_s},$ $A_2 = \frac{A_1\alpha_{s+1}}{(s-2)\alpha_s}$, we obtain $\alpha_s '=1$ and $\alpha_{s+1}'=0.$ Hence, without loss of generality, we may assume that $\alpha_s =1,$ $\alpha_{s+1}=0.$ This establishes the base step of the induction.

		Assume now that $\alpha_{s}=1$ and for some $k \ge 1,$ we have $\alpha_{s+1}= \dots= \alpha_{s+k}=0.$ Choosing $A_1 = 1,$ $A_2 =\dots=A_{k+1}=0$
		in~\eqref{inteq1}, we have that
		\begin{center}
			$\alpha_s '=1, \quad \alpha_i '=0, \ s+1 \leq i \leq s+k, \quad \alpha_{s+k+1}'=\alpha_{s+k+1} - A_{k+2}.$
		\end{center}
		Taking $A_{k+2}=\alpha_{s+k+1},$ we obtain $\alpha_{s+k+1}'=0.$

		Thus, by induction, any algebra from the class $Id(\alpha_1,\dots,\alpha_n )$ with $\alpha_1 =0$ and $\alpha_s \neq 0$ for some $ 3 \le s \le n$ is isomorphic to an algebra with parameters
		\begin{center}
			$\alpha_{1}'=0,\quad \alpha_2 ',\quad \alpha_s '=1, \quad \alpha_i ' = 0 \quad \text{for all}\quad i \quad\text{such that} \quad i\neq s\quad \text{and}\quad 3 \le i \le n.$
		\end{center}
		Thus, $\mathfrak{L}^{\alpha_1,\dots,\alpha_n}$ is isomorphic to $\mathbb{B}_s (\alpha)$ for some $s.$
	\end{proof}

	Note that if $\alpha_1 =0$ and $\alpha_s = 0$ for all $3 \le s \le n,$ then $\alpha_2 '=\alpha_2$ and we obtain the following one-parameter family of algebras:
	\[
		\mathbb{B}_2 (\alpha):
		\left\{
		\begin{array}{ll}
			e_i \cdot e_j =e_{i+j}, & 2\leq i+j \leq n, \\[1mm]
			e_i \star e_j = \alpha e_{i+j}, & 2\leq i+j \leq n,
		\end{array}
		\right.
	\]
	where we can define $\star$ multiplication with $\cdot$ in the following form: $e_i \star e_j = \alpha (e_i \cdot e_j ).$

	\begin{theorem}\label{int} Let $ \star$ be an $\id$-matching structure with $\cdot$ on the algebra $\mathbb{C}_n [x]$. Then, $(\mathbb{C}_n [x], \cdot, \star)$ is isomorphic to one of the following pairwise non-isomorphic algebras:
		\[
			\mathbb{B}_1, \quad \mathbb{B}_s (\alpha), \quad 2\leq s\leq n.
		\]
	\end{theorem}

	\begin{proof}
		Since any $\id$-matching structure on the algebra $(\mathbb{C}_n [x],\cdot)$ has multiplication given by~\eqref{int7}, we distinguish cases according to the relation~\eqref{alpha_s}:
		\begin{itemize}
			\item If $\alpha_1\neq0,$ then by Lemma~\ref{intlem1}, the algebra $(\mathbb{C}_n [x], \cdot, \star)$ is isomorphic to $\mathbb{B}_{1};$
			\item If $\alpha_1 =0$ and $\alpha_s\neq0$ for some $3\leq s\leq n,$ then by Lemma~\ref{intlem2}, the algebra $(\mathbb{C}_n [x], \cdot, \star)$ is isomorphic to the algebra $\mathbb{B}_{s}(\alpha);$

			\item If $\alpha_1 =0$ and $\alpha_i =0$ for all $3\leq i\leq n,$ then the algebra $(\mathbb{C}_n [x], \cdot, \star)$ is isomorphic to the algebra $\mathbb{B}_{2}(\alpha).$
        \qedhere
		\end{itemize}
	\end{proof}

	\section{\texorpdfstring{$(12)$-matching structure on $\mathbb{C}_n [x]$}{(12)-matching structure on \mathbb{C}\_n [x]}}\label{12-matching-section}

	\begin{lemma}\label{Lem3.1}
		Let $\star$ be a~$(12)$-matching structure on the algebra $(\mathbb{C}_n [x],\cdot)$. Then the following relations hold:
		\[
			\begin{array}{lll}
				e_{i} \star e_j = e_1 \star e_{i+j-1}, & 2\leq i+j \leq n+1, \\[1mm]
				e_i \star e_j =0, & n+2 \leq i+j \leq 2n.
			\end{array}
		\]
	\end{lemma}

	\begin{proof}
		Assume that the products $\cdot$ and $\star$ are $(12)$-matching. Using the identity~\eqref{matching-ass-identity} and multiplication~\eqref{mu_0}, we have
		\begin{center}
			$e_{i+1} \star e_j =(e_i \cdot e_1 )\star e_j = e_i \star (e_1\cdot e_j )=e_i \star e_{j+1}.$
		\end{center}
		By induction on $i,$ this implies $e_{i} \star e_j = e_1 \star e_{i+j-1}$ for $2\leq i+j \leq n+1.$

		Since $e_{i} \cdot e_j =0$ whenever $ i+j \geq n+1,$ considering
		\begin{center}
			$e_{i+1}\star e_j = (e_1 \cdot e_{i})\star e_j = e_1 \star (e_{i} \cdot e_j ),$
		\end{center}
		we get that $e_{i} \star e_j =0$ for $i+j \geq n+2.$
	\end{proof}

	\begin{theorem}\label{th15}
		Let $ \star$ be a~$(12)$-matching structure on the algebra $(\mathbb{C}_n [x], \cdot)$. Then, either $\cdot$ and $\star$ are $\id$-matching or multiplications have the form
		\begin{align*}
			\textbf{(12)}&(\alpha_1,\dots, \alpha_{n-1}, \beta_1,\dots, \beta_{n-1}): \\
			&\begin{cases}
				e_i \cdot e_j =e_{i+j}, & 2\leq i+j \leq n, \\[1mm]
				e_i \star e_j = \sum_{t=i+j-1}^{n-1}\alpha_{t-i-j+2}e_t +\beta_{i+j-1} e_n, & 2\leq i+j \leq n.
			\end{cases}
		\end{align*}
	\end{theorem}

	\begin{proof}
		Assume that the products $\cdot$ and $\star$ are $ (12)$-matching. Set
		\begin{center}
			$e_1\star e_i =\sum_{t=1}^{n}\alpha_{i,t}e_t, \quad 1 \leq i \leq n-1.$
		\end{center}
		We first apply the identity~\eqref{matching-ass-identity} to the triple
		$(e_1, e_1, e_i )$ where $2\leq i \leq n:$
		\begin{center}
			$(e_1 \star e_1 )\cdot e_i =\sum_{t=1}^{n}\alpha_{1,t}e_t \cdot e_i =\sum_{t=i+1}^{n}\alpha_{1,t-i}e_t.$
		\end{center}
		On the other hand,
		\begin{center}
			$e_1\cdot (e_1\star e_i )= e_1\cdot \sum_{t=1}^{n}\alpha_{i,t}e_t = \sum_{t=2}^{n}\alpha_{i,t-1}e_t.$
		\end{center}
		Comparing coefficients of the basis elements, we obtain
		\begin{center}
			$\alpha_{i,s}=0, \quad 1\leq s\leq i-1;$ \quad $\alpha_{i,t}=\alpha_{1,t-i+1}, \quad i\leq t\leq n-1, \ 2\leq i\leq n.$
		\end{center}
		Therefore,
		\begin{center}
			$ e_1 \star e_i = \sum_{t=i}^{n-1}\alpha_{1,t-i+1}e_t + \alpha_{i,n} e_n, \quad 1\leq i\leq n.$
		\end{center}
		Introduce the notation $\alpha_t:=\alpha_{1,t},$ $1\leq t\leq n-1,$ $\beta_{i}:=\alpha_{i,n},$ $1\leq i\leq n.$
		Using Lemma~\ref{Lem3.1}, we obtain
		\begin{equation}\label{eq11}
			e_i \star e_j = \sum_{t=i+j-1}^{n-1}\alpha_{t-i-j+2}e_t + \beta_{i+j-1} e_n, \qquad 2 \leq i+j \leq n+1.
		\end{equation}
		It remains to impose associativity. Consider
		\begin{align*}
			e_1 \star (e_1 \star e_i ) &= \sum_{m=i}^{n-1} \sum_{t=1}^{m-i+1} \alpha_{t} \alpha_{m-t-i+2}e_m + \sum_{t=i}^{n-1} \alpha_{t-i+1}\beta_{t}e_n + \beta_i \beta_n e_n;\\
			(e_1 \star e_1 ) \star e_i
			&=\sum_{m=i}^{n-1} \sum_{t=1}^{m-i+1} \alpha_{t} \alpha_{m-t-i+2}e_m + \sum_{t=i}^{n-1} \alpha_{t-i+1}\beta_{t}e_n + \alpha_{n-i+1}\beta_{n}e_n.
		\end{align*}
		It shows that the coefficients of $e_m$ for
		$m \leq n-1$ coincide, and comparing the coefficients of $e_n$, we obtain
		\[
			\beta_n (\beta_i -\alpha_{n-i+1})=0, \quad 2\leq i\leq n.
		\]
		If $\beta_i = \alpha_{n-i+1}$ for all $2\leq i\leq n,$ then we get that the equality~\eqref{eq11} coincides with the equality~\eqref{int6}. Thus, in this case, we obtain that $\cdot$ and $\star$ are $\id$-matching.

		If $\beta_i \neq \alpha_{n-i+1}$ for some $2\leq i\leq n,$ then $\beta_n =0$ and we have the family of algebras $\textbf{(12)}(\alpha_1,\dots, \alpha_{n-1}, \beta_1,\dots, \beta_{n-1}).$
	\end{proof}

	Similarly to Lemma~\ref{Lemma10}, we provide a~criterion for determining when two algebras from the class $\textbf{(12)}(\alpha_1,\dots, \alpha_{n-1}, \beta_1,\dots, \beta_{n-1})$ are isomorphic.

	\begin{lemma}\label{lem16} 
    Two algebras, $\mathfrak{S}^{\alpha_1,\dots,\alpha_{n-1}, \beta_{1},\dots, \beta_{n-1}}$ and $\mathfrak{S}'^{\alpha_1 ',\dots,\alpha_{n-1}', \beta_{1}',\dots, \beta_{n-1}'}$, from the same class $\textbf{(12)}(\alpha_1,\dots, \alpha_{n-1}, \beta_1,\dots, \beta_{n-1})$ are isomorphic if and only if there exist $A_1 \in \mathbb{C}^*,$ $A_2, \dots, A_n \in \mathbb{C},$ such that the following relations holds:
    \begin{gather*}
        \sum_{i=1}^t\sum_{k_1 +\cdots+k_i =t} A_{k_1} \dots A_{k_i}\alpha_i ' = \sum_{i=1}^{t} \sum_{j=1}^{t-i+1} A_i A_j \alpha_{t-i-j+2}, \\
        \sum_{i=1}^{n-1} \sum_{k_1 +\cdots+k_i =n} A_{k_1} \dots A_{k_i} \alpha_i ' + A_1^n \beta'_{1} = \sum_{2\leq i+j\leq n} A_i A_j \beta_{i+j-1}. \\
        \sum_{m=i}^{n-1} \sum_{k_1 +\cdots+k_m =n} \alpha_{m-i+1}' A_{k_1} \dots A_{k_m} +\beta'_{i} A_1^n = \sum_{j=1}^{n-i} \sum_{m=i}^{n-j} A_j \sum_{k_1 + \dots +k_i =m} A_{k_1} \dots A_{k_i} \beta_{j+m-1},
    \end{gather*}
    for all $1 \leq t \leq n-1$ and $2 \leq i \leq n-1$.
	\end{lemma}

	\begin{proof}
		Let $\varphi: \mathfrak{S}'^{\alpha_1 ',\dots,\alpha_{n-1}', \beta_{1}',\dots, \beta_{n-1}'} \rightarrow \mathfrak{S}^{\alpha_1,\dots,\alpha_{n-1}, \beta_{1},\dots, \beta_{n-1}}$ be an isomorphism of algebras. Since, $\varphi$ is an automorphism of the algebra $(\mathbb{C}_n [x],\cdot),$ it follows from Theorem~\ref{3.1} that $\varphi$ has the form given in~\eqref{Aut}. Consider the identity $\varphi(e_i \star' e_j ) = \varphi(e_i )\star \varphi(e_j ).$

		First, we compute
        \begin{align*}
            \varphi(e_1 \star' e_1 ) 
            &= \varphi\Big(\sum_{i=1}^{n-1}\alpha_{i}'e_i + \beta'_{1}e_n\Big) \\
            &=\sum_{i=1}^{n-1} \alpha_i '\sum_{t=i}^{n}\sum_{k_1 +\dots +k_i =t} A_{k_1}\dots A_{k_i}e_t + \beta'_{1} A_1^n e_n \\
            &= \sum_{t=1}^{n-1} \sum_{i=1}^{t} \sum_{k_1 +\dots +k_i =t}\alpha_i ' A_{k_1} \dots A_{k_i}e_t + \left( \sum_{i=1}^{n-1} \sum_{k_1 +\cdots+k_i =n} \alpha_i ' A_{k_1} \dots A_{k_i} +\beta'_{1} A_1^n\right) e_n.
        \end{align*}
		On the other hand,
		\[
			\begin{array}{lcl}
				\varphi(e_1 ) \star \varphi(e_1 ) &=& \sum_{i=1}^n A_i e_i \star \sum_{j=1}^n A_j e_j =
				\sum_{i=1}^n \sum_{j=1}^n A_i A_j (e_i \star e_j )\\[1mm]
				&=&
				\sum_{i=1}^n \sum_{j=1}^n A_i A_j \left( \sum_{t=i+j-1}^{n-1} \alpha_{t-i-j+2} e_t + \beta_{i+j-1}e_n \right) \\[1mm]
				&=&
				\sum_{i=1}^n \sum_{j=1}^n \sum_{t=i+j-1}^{n-1} A_i A_j \alpha_{t-i-j+2} e_t + \sum_{2\leq i+j\leq n} A_i A_j \beta_{i+j-1}e_n \\[1mm]
				&=&
				\sum_{t=1}^{n-1}\sum_{i=1}^{t} \sum_{j=1}^{t-i+1} A_i A_j \alpha_{t-i-j+2} e_t + \sum_{2\leq i+j\leq n} A_i A_j \beta_{i+j-1}e_n.
			\end{array}
		\]
		By comparing the coefficients of the basis elements, we obtain
		\begin{gather*}
			\sum_{i=1}^t\sum_{k_1 +\cdots+k_i =t} A_{k_1} \dots A_{k_i}\alpha_i ' = \sum_{i=1}^{t} \sum_{j=1}^{t-i+1} A_i A_j \alpha_{t-i-j+2}, \ \ 1\leq t\leq n-1. \\
			\sum_{i=1}^{n-1} \sum_{k_1 +\cdots+k_i =n} A_{k_1} \dots A_{k_i} \alpha_i ' + A_1^n \beta'_{1} = \sum_{2\leq i+j\leq n} A_i A_j \beta_{i+j-1}.
		\end{gather*}
		Next, we consider
		\[
			\begin{array}{lcl}
				\varphi(e_1 \star' e_i ) &= & \sum_{m=i}^{n-1}\alpha_{m-i+1}'\varphi(e_m )+ \beta'_{i}\varphi(e_n ) =\sum_{m=i}^{n-1} \alpha_{m-i+1}'\sum_{t=m}^{n}G_t^{(m)} e_t + \beta'_{i} A_1^n e_n \\[1mm]
				&=& \sum_{t=i}^{n-1} \sum_{m=i}^{t} \alpha_{m-i+1}' G_t^{(m)}e_t + \left( \sum_{m=i}^{n-1} \alpha_{m-i+1}' G_n^{(m)} +\beta'_{i} A_1^n\right) e_n,
			\end{array}
		\]
		where $G_m^{(i)}= \sum_{k_1 + \dots +k_i =m} A_{k_1} \dots A_{k_i}$.

		On the other hand,
        \begin{align*}
            \varphi(e_1 ) \star \varphi(e_i )
            &= \sum_{j=1}^n A_j e_j \star \sum_{m=i}^n G_m^{(i)} e_m \\
            &= \sum_{j=1}^n \sum_{m=i}^n A_j G_m^{(i)} (e_j \star e_m ) \\
            &= \sum_{j=1}^n \sum_{m=i}^n A_j G_m^{(i)} \left( \sum_{t=j+m-1}^{n-1} \alpha_{t-j-m+2} e_t + \beta_{j+m-1}e_n \right) \\
            &= \sum_{j=1}^n \sum_{m=i}^n \sum_{t=j+m-1}^{n-1} A_j G_m^{(i)} \alpha_{t-j-m+2} e_t + \sum_{j=1}^{n-i} \sum_{m=i}^{n-j} A_j G_m^{(i)} \beta_{j+m-1}e_n \\
            &= \sum_{t=i}^{n-1} \sum_{j=1}^{t-i+1} \sum_{m=i}^{t-j+1} A_j G_m^{(i)} \alpha_{t-j-m+2} e_t + \sum_{j=1}^{n-i} \sum_{m=i}^{n-j} A_j G_m^{(i)} \beta_{j+m-1}e_n.
        \end{align*}
		By comparing the coefficients of the basis element $e_n$, we obtain
		\[
			\sum_{m=i}^{n-1} \sum_{k_1 +\cdots+k_m =n} \alpha_{m-i+1}' A_{k_1} \dots A_{k_m} +\beta'_{i} A_1^n = \sum_{j=1}^{n-i} \sum_{m=i}^{n-j} A_j \sum_{k_1 + \dots +k_i =m} A_{k_1} \dots A_{k_i} \beta_{j+m-1}.
		\]
		Similarly to Lemma~\ref{Lemma10}, all remaining products satisfy
		$0 = 0.$
	\end{proof}

	From Theorem~\ref{th15}, we obtain that $I = \operatorname{span}\{e_n\}$ is a~central ideal of the $(12)$-matching structure on $(\mathbb{C}_n [x], \cdot).$ Moreover, the quotient algebra by $I$ is an $\id$-matching structure on $(\mathbb{C}_{n-1}[x], \cdot).$ Thus,
	\[
		\textbf{(12) }(\alpha_1,\dots, \alpha_{n-1}, \beta_1,\dots, \beta_{n-1}) / I \cong Id(\alpha_1, \dots,\alpha_{n-1}).
	\]
	Hence, using Theorem~\ref{int}, we see that any algebra in the $\textbf{(12) }(\alpha_1,\dots, \alpha_{n-1}, \beta_1,\dots, \beta_{n-1})$ class has multiplication in one of the following forms:
	\[
		\textbf{(12)}_1 (\beta_1,\dots,\beta_{n-1}):\left\{
		\begin{array}{ll}
			e_i \cdot e_j =e_{i+j}, & 2\leq i+j \leq n, \\[1mm]
			e_i \star e_j =e_{i+j-1}+\beta_{i+j-1}e_n, & 2\leq i+j \leq n,
		\end{array}
		\right.
	\]
	where $(\beta_2, \dots, \beta_{n-1}) \neq (0, \dots, 0)$;
	\[
		\textbf{(12)}_2 (\alpha,\beta_1,\dots,\beta_{n-1}):\left\{
		\begin{array}{ll}
			e_i\cdot e_j =e_{i+j}, & 2\leq i+j \leq n, \\[1mm]
			e_i \star e_j = \alpha e_{i+j} +\beta_{i+j-1}e_n, & 2\leq i+j \leq n-1,\\[1mm]
			e_i \star e_j = \beta_{n-1}e_n, & i+j=n,
		\end{array}
		\right.
	\]
	where $(\beta_2, \dots, \beta_{n-2}, \beta_{n-1}) \neq (0, \dots, 0, \alpha)$;
	\begin{align*}
		\textbf{(12)}_s &(\alpha, \beta_1,\dots,\beta_{n-1}): \\
        &\left\{
		\begin{array}{ll}
			e_i\cdot e_j =e_{i+j}, & 2\leq i+j \leq n, \\[1mm]
			e_i \star e_j = \alpha e_{i+j} + e_{s+i+j-2}+\beta_{i+j-1}e_n, & 2\leq i+j \leq n-s+1, \\[1mm]
			e_i \star e_j = \alpha e_{i+j} +\beta_{i+j-1}e_n, & n-s+2\leq i+j \leq n-1,\\[1mm]
			e_i \star e_j = \beta_{n-1}e_n, & i+j=n, \ 3\leq s\leq n-1,
		\end{array}
		\right.
	\end{align*}
	where $(\beta_2, \dots,\beta_{n-s}, \beta_{n-s+1}, \beta_{n-s+2}, \dots, \beta_{n-2}, \beta_{n-1}) \neq (0, \dots, 0, 1_{n-s+1}, 0, \dots, 0, \alpha).$

	Now we will examine these families, treating each algebra individually.

	\begin{lemma}\label{12-lem-alpha_1} Let $\mathfrak{S}^{\beta_{1},\dots, \beta_{n-1}}$ be an algebra from the class $\textbf{(12)}_1 (\beta_{1},\dots,\beta_{n-1})$. Then it is isomorphic to the algebra:
		\[
			\mathbb{A}_1 (\beta_2,\dots,\beta_{n-1}):\left\{
			\begin{array}{ll}
				e_i \cdot e_j =e_{i+j}, & 2\leq i+j \leq n, \\[1mm]
				e_1 \star e_1 =e_{1}, \\[1mm]
				e_i \star e_j =e_{i+j-1}+\beta_{i+j-1}e_n, & 3\leq i+j \leq n.
			\end{array}
			\right.
		\]
	\end{lemma}

	\begin{proof}
		Let $\mathfrak{S}^{\beta_{1},\dots, \beta_{n-1}}$ be an algebra from the class $\textbf{(12)}_1 (\beta_{1},\dots,\beta_{n-1})$. It means that
		\begin{center}
			$\alpha_1 =1$ and $\alpha_2 =\dots =\alpha_{n-1}=0.$
		\end{center}
		Then, from Lemma~\ref{lem16}, we have that $A_1 =1,$ $A_i =0$ for $2 \leq i \leq n-1.$
		Now, by choosing $A_n = \beta_1,$ we obtain $\beta_1 ' = 0$. Substituting these values into the relation given in Lemma~\ref{lem16}, we obtain $\beta_i ' = \beta_i$ for $2 \le i \le n-1.$

		Thus, we have that
		$\mathfrak{S}^{ \beta_{1},\dots, \beta_{n-1}}\cong \mathbb{A}_1 (\beta_{2},\dots,\beta_{n-1})$.
	\end{proof}

	\begin{lemma}\label{12-lem-alpha_2}
		Let $\mathfrak{S}^{\alpha, \beta_{1},\dots, \beta_{n-1}}$ be an algebra from the class $\textbf{(12)}_2 (\alpha, \beta_{1},\dots,\beta_{n-1})$.
		\begin{enumerate}
			\item If $\beta_{n-1} \neq \alpha,$ then it is isomorphic to the algebra
			\[
				\mathbb{A}_2 (\alpha,\beta)_{\beta \neq \alpha}:\left\{
				\begin{array}{ll}
					e_i\cdot e_j =e_{i+j}, & 2\leq i+j \leq n, \\[1mm]
					e_i\star e_j =\alpha e_{i+j}, & 2\leq i+j \leq n-1, \\[1mm]
					e_i \star e_j =\beta e_{n}, & i+j = n.
				\end{array}
				\right.
			\]
			\item If $\beta_{n-1} = \alpha, $ $\beta_{n-i}=0,$ $2\leq i\leq r-1$ and $\beta_{n-r} \neq 0, \ 2 \le r \le {n-1},$ then it is isomorphic to the algebra
			\[
				\mathbb{A}_{3,r}(\alpha):\left\{
				\begin{array}{ll}
					e_i\cdot e_j =e_{i+j}, & 2\leq i+j \leq n, \\[1mm]
					e_i\star e_j =\alpha e_{i+j}, & 2\leq i+j \leq n-r, \\[1mm]
					e_i\star e_j =\alpha e_{n-r+1}+e_n, & i+j =n-r+1, \\[1mm]
					e_i\star e_j =\alpha e_{i+j}, & n-r+2 \leq i+j \leq n,
				\end{array}
				\right. \ \ 2\leq r\leq n-1.
			\]
		\end{enumerate}
	\end{lemma}

	\begin{proof}
		Let $\mathfrak{S}^{\alpha, \beta_{1},\dots, \beta_{n-1}}$ be an algebra from the class $\textbf{(12)}_2 (\alpha, \beta_{1},\dots,\beta_{n-1})$. It means that
		\begin{center}
			$\alpha_1 =0,$ $\alpha_3 =\dots =\alpha_{n-1}=0.$
		\end{center}

		\begin{enumerate}
			\item[\textbf{Case 1:}] $\beta_{n-1} \neq \alpha.$ From Lemma~\ref{lem16}, we have that
			\begin{center}
				$\beta_{n-1}'=\beta_{n-1}, \quad \beta_{n-2}'=\frac{1}{A_1} \beta_{n-2} +\frac{A_2}{A_1^2} (n-1) (\beta_{n-1}-\alpha).$
			\end{center}
			By choosing $A_2 = \frac{A_1 \beta_{n-2}}{(n-1)(\alpha-\beta_{n-1})}$, we obtain $\beta_{n-2}'=0.$ Hence, without loss of generality, we may assume that $\beta_{n-2}=0.$ We now prove by induction that one may further assume that $\beta_{n-i} = 0$ for all $2 \le i \le n-1$.

			The base case $i=2$ has already been established.

			Consider $i=3$. Assume $\beta_{n-2}=0$, then choosing $A_2 = 0$ in Lemma~\ref{lem16}, we obtain
			\begin{center}
				$\beta_{n-2}'=0, \quad \beta_{n-3}'=\frac{1}{A_1^2} \beta_{n-3} +\frac{A_3}{A_1^3} (n-2) (\beta_{n-1}-\alpha).$
			\end{center}
			By taking $A_3 =\frac{A_1 \beta_{n-3}}{(n-2)(\alpha-\beta_{n-1})},$ we obtain $\beta_{n-3}'=0.$
			Assume now that for some $k \ge 3,$ we have $\beta_{n-2}= \dots= \beta_{n-k+1}=0,$ then choosing $A_2 =\dots=A_{k-1}=0$
			in Lemma~\ref{lem16}, we have that
			\begin{center}
				$\beta_{n-i}'=0, \ 2 \leq i \leq k-1, \quad \beta_{n-k}'=\frac{1}{A_1^{k-1}} \beta_{n-k} +\frac{A_k}{A_1^k} (n-k+1) (\beta_{n-1}-\alpha).$
			\end{center}
			Choosing $A_k =\frac{A_1 \beta_{n-k}}{(n-k+1)(\alpha-\beta_{n-1})},$ we obtain $\beta_{n-k}'=0.$
			Thus, by induction, any algebra from the class $\textbf{(12)}_2 (\alpha, \beta_{1},\beta_{2},\dots,\beta_{n-1})$ is isomorphic to the algebra with parameters $\beta_{n-1}' = \beta$ and $\beta_{n-i}' = 0$ for all $2 \le i \le n-1$, that is, to $\mathbb{A}_2 (\alpha,\beta)_{\beta \neq \alpha}$.

			\item[\textbf{Case 2:}] $\beta_{n-1} = \alpha.$ Then there exist $2 \le r \le {n-1}$ such that $\beta_{n-r} \neq 0.$ Assume that
			$\beta_{n-i}=0$ for all $2\leq i\leq r-1$ and $\beta_{n-r} \neq 0.$ Choosing $A_1 = \sqrt[r-1]{\beta_{n-r}},$ $A_2 = -\frac{\beta_{n-r-1}}{n-r}$ in Lemma~\ref{lem16}, we obtain $\beta_{n-r}'=1,$ $\beta_{n-r-1}'=0.$

			Hence, without loss of generality, we may assume that $\beta_{n-r}=1,$ $\beta_{n-r-1}=0.$
			We now prove by induction that one may further assume $\beta_{n-i} = 0$ for all $r+1 \le i \le n-1$.

			The base case $i=r+1$ has already been established. Consider $i=r+2$, that is $\beta_{n-r}=1,$ $\beta_{n-r-1}=0.$ Choosing $A_1 = 1,$ $A_2 = 0$ in Lemma~\ref{lem16}, we obtain
			\begin{center}
				$\beta_{n-r}'=1, \quad \beta_{n-r-1}'=0, \quad \beta_{n-r-2}'=\beta_{n-r-2} +(n-r-1) A_{3}.$
			\end{center}
			By taking $A_3 =-\frac{\beta_{n-r-2}}{n-r-1},$ we obtain $\beta_{n-r-2}'=0.$
			Assume now that for some $k \ge r+2,$ we have
			\begin{center}
				$\beta_{n-1} = \alpha,$ $\beta_{n-r} =1,$ $\beta_{n-2}= \dots= \beta_{n-r+1} = \beta_{n-r-1}= \dots =\beta_{n-k+1}=0.$
			\end{center}
			Then choosing
			$A_1 =1,$ $A_2 =\dots=A_{k-1}=0$
			in Lemma~\ref{lem16}, we have that
			\begin{center}
				$\beta_{n-i}'=0, \ r+1 \leq i \leq k-1, \quad \beta_{n-k}'=\beta_{n-k} +(n-k+1) A_k.$
			\end{center}
			Choosing $A_k =-\frac{\beta_{n-k}}{n-k+1},$ we obtain $\beta_{n-k}'=0.$
			Thus, by induction, any algebra from the class $\textbf{(12)}_2 (\alpha, \beta_{1},\beta_{2},\dots,\beta_{n-1})$ is isomorphic to the algebra with parameters $\beta_{n-1}'= \alpha,$ $\beta_{n-r}' = 1$ and $\beta_{n-i}' = 0$ for all $2 \le i \ (i \neq r) \le n-1$, that is, to $\mathbb{A}_{3,r}(\alpha),$ $2\leq r\leq n-1.$
		\end{enumerate}
	\end{proof}

	\begin{lemma}\label{12-lem-alpha_2 -s} Let $\mathfrak{S}^{\alpha, \beta_{1},\dots, \beta_{n-1}}$ be an algebra in the class $\textbf{(12)}_s (\alpha, \beta_{1},\dots,\beta_{n-1})$ for some $3\leq s\leq n-1.$

		1. If $\beta_{n-1} \neq \alpha,$ then it is isomorphic to the algebra
		\begin{align*}
			\mathbb{A}_{4,s}&(\alpha,\beta_s,\dots,\beta_{n-1}): \\
            &\left\{
			\begin{array}{ll}
				e_i\cdot e_j =e_{i+j}, & 2\leq i+j \leq n, \\[1mm]
				e_i\star e_j =\alpha e_{i+j}+ e_{s+i+j-2}, & 2\leq i+j \leq s, \\[1mm]
				e_i\star e_j =\alpha e_{i+j}+ e_{s+i+j-2} + \beta_{i+j-1}e_n, & s+1 \leq i+j \leq n-s+1, \\[1mm]
				e_i\star e_j =\alpha e_{i+j} + \beta_{i+j-1}e_n, & n-s+2 \leq i+j \leq n-1, \\[1mm]
				e_i \star e_j =\beta_{n-1} e_{n}, & i+j = n.
			\end{array}
			\right.
		\end{align*}
		2. If $\beta_{n-1} = \alpha,$ $\beta_{n-i} = 0,$ $ 2 \le i \le {r-1},$ $\beta_{n-r} \neq 0,$ $2 \leq r \leq s-2,$ then it is isomorphic to the algebra
		\begin{align*}
			\mathbb{A}_{5,s,r}&(\alpha,\beta_{s-r+1},\dots,\beta_{n-r}): \\
            &\left\{
			\begin{array}{ll}
				e_i\cdot e_j = e_{i+j}, & 2\leq i+j \leq n, \\[1mm]
				e_i\star e_j =\alpha e_{i+j} + e_{s+i+j-2}, & 2\leq i+j \leq s-r+1, \\[1mm]
				e_i\star e_j =\alpha e_{i+j} + e_{s+i+j-2} +\beta_{i+j-1} e_n, & s-r+2 \leq i+j \leq n-s+1, \\[1mm]
				e_i\star e_j =\alpha e_{i+j}+ \beta_{i+j-1} e_n, & n-s+2 \leq i+j \leq n-r+1, \\[1mm]
				e_i\star e_j =\alpha e_{i+j}, & n-r+2 \leq i+j \leq n,
			\end{array}
			\right.
		\end{align*}
		where we assume what $e_t =0,$ if $t>n.$

		3. If $\beta_{n-1} = \alpha, \ \beta_{n-i} = 0, \ 2 \leq i \leq s-2, \ \beta_{n-s+1} \neq \frac{s}{2},$ then it is isomorphic to the algebra:
        \begin{align*}
			\mathbb{A}_{6,s}&(\alpha,\beta_2,\dots,\beta_{n-s+1}): \\
            &\left\{
			\begin{array}{ll}
				e_i\cdot e_j = e_{i+j}, & 2\leq i+j \leq n, \\[1mm]
				e_1\star e_1 =\alpha e_{2} + e_{s}, \\[1mm]
				e_i\star e_j =\alpha e_{i+j} + e_{s+i+j-2} +\beta_{i+j-1}e_n, & 3\leq i+j \leq n-s+1, \\[1mm]
				e_i\star e_j =\alpha e_{n-s+2} +\beta_{n-s+1}e_n, & i+j = n-s+2, \\[1mm]
				e_i\star e_j =\alpha e_{i+j}, & n-s+3 \leq i+j \leq n.
			\end{array}
			\right.
		\end{align*}
		4. If $\beta_{n-1} = \alpha,$ $\beta_{n-i} = 0,$ $2 \leq i \leq s-2,$ $\beta_{n-s+1}=\frac{s}{2},$ then it is isomorphic to the algebra:
		\[
			\mathbb{A}_{7,s}(\alpha,\beta_1,\dots,\beta_{n-s}):\left\{
			\begin{array}{ll}
				e_i\cdot e_j = e_{i+j}, & 2\leq i+j \leq n, \\[1mm]
				e_i\star e_j =\alpha e_{i+j} + e_{s+i+j-2} +\beta_{i+j-1}e_n, & 2\leq i+j \leq n-s+1, \\[1mm]
				e_i\star e_j =\alpha e_{n-s+2} +\frac{s}{2}e_n, & i+j = n-s+2, \\[1mm]
				e_i\star e_j =\alpha e_{i+j}, & n-s+3 \leq i+j \leq n.
			\end{array}
			\right.
		\]
	\end{lemma}

	\begin{proof}
		Let $\mathfrak{S}^{\alpha, \beta_{1},\dots, \beta_{n-1}}$ be an algebra from the class $\textbf{(12)}_s (\alpha, \beta_{1},\beta_{2},\dots,\beta_{n-1}).$ It means that
		$\alpha_1 =0,$ $\alpha_2 =\alpha,$ $\alpha_s =1,$ $\alpha_i =0$ for all $3 \leq i \ (i \neq s) \leq n-1.$
		Moreover, $A_1 =1,$ $A_i =0$ for all $2 \leq i \leq n-s.$
		\begin{enumerate}
			\item[\textbf{Case 1:}] $\beta_{n-1} \neq \alpha, $ by choosing $A_{n-s+1} = \frac{\beta_{s-1}}{s(\alpha -\beta_{n-1})}$ in Lemma~\ref{lem16}, we obtain $\beta_{s-1}'=0.$ Hence, without loss of generality, we may assume that $\beta_{s-1}=0.$ We now prove by induction that one may further assume that $\beta_{s-i} = 0$ for all $1 \le i \le s-1$.

			The base case $i=1$ has already been established. Consider $i=2$. Assume $\beta_{s-1}=0$. Choosing $A_{n-s+1} = 0$ in Lemma~\ref{lem16}, we obtain
			\begin{center}
				$\beta_{j}'=\beta_{j}, \ s \leq j \leq n-1, \quad \beta_{s-1}'=0, \quad \beta_{s-2}'=\beta_{s-2} +(s-1) A_{n-s+2} (\beta_{n-1}-\alpha).$
			\end{center}
			By taking $A_{n-s+2}=\frac{\beta_{s-2}}{(s-1)(\alpha-\beta_{n-1})},$ we have $\beta_{s-2}'=0.$
			Assume now, for some $k \ge 2$, that we have $\beta_{s-1}= \dots= \beta_{s-k+1}=0$. Then, choosing $A_{n-s+1} =\dots=A_{n-s+k-1}=0$
			in Lemma~\ref{lem16}, we get that
			\begin{center}
				$\beta_{s-i}'=0, \ 2 \leq i \leq k-1, \quad \beta_{s-k}'=\beta_{s-k} + (s-k+1) A_{n-s+k} (\beta_{n-1}-\alpha).$
			\end{center}
			Choosing $A_{n-s+k}=\frac{\beta_{s-k}}{(s-k+1)(\alpha - \beta_{n-1})},$ we obtain $\beta_{s-k}'=0.$
			Thus, by induction, we have that in this case any algebra from the class $\textbf{(12)}_s (\alpha, \beta_{1},\beta_{2},\dots,\beta_{n-1})$ is isomorphic to the algebra with parameters $\alpha_2 ' = \alpha,$ $\alpha_s '=1,$ $\beta_{i}' = 0$ for all $1 \le i \le s-1$, and $\beta_j '=\beta_j$ for all $s \leq j \leq n-1$ that is, to $\mathbb{A}_{4,s}(\alpha,\beta_s,\dots,\beta_{n-1}).$

			\item[\textbf{Case 2:}] $\beta_{n-1} = \alpha.$ Then there exist $2 \le r \le {s-2}$ such that $\beta_{n-r} \neq 0.$ Assume that
			$\beta_{n-i}=0$ for all $2\leq i\leq r-1$ and $\beta_{n-r} \neq 0.$ Choosing $A_{n-s+1} = -\frac{\beta_{s-r}}{(s-r+1)\beta_{n-r} }$ in Lemma~\ref{lem16}, we obtain $\beta_{s-r}'=0.$
			Hence, without loss of generality, we may assume that $\beta_{s-r}=0.$
			We now prove by induction that one may further assume $\beta_{s-j-r+1} = 0$ for all $1 \le j \le s-r$.

			The base case $j=1$ has already been established. Consider $j=2$. Assume $\beta_{s-r}=0.$ Choosing $A_{n-s+1} = 0$ in Lemma~\ref{lem16}, we obtain
			\begin{center}
				$\beta_{s-r}'=0, \quad \beta_{s-r-1}'=\beta_{s-r-1} +(s-r) A_{n-s+2} \beta_{n-r}.$
			\end{center}
			By taking $A_{n-s+2}=-\frac{\beta_{s-r-1}}{(s-r) \beta_{n-r}},$ we obtain $\beta_{s-r-1}'=0.$
			Assume now that for some $k \ge 3,$ we have
			\begin{center}
				$\beta_{n-1} = \alpha,$ $\beta_{n-2}= \dots= \beta_{n-r+1} = \beta_{n-r-1}= \dots =\beta_{n-k+1}=0,$
			\end{center}
			then choosing $A_1 =1,$ $A_2 =\dots=A_{n-s+2}=0$
			in Lemma~\ref{lem16}, we have that
			\begin{center}
				$\beta_{s-i}'=0, \ r \leq i \leq k-1, \quad \beta_{s-r-k+1}'=\beta_{s-r-k+1} +(s-r-k+2) A_{n-s+k} \beta_{n-r}.$
			\end{center}
			Choosing $A_{n-s+k}=-\frac{\beta_{s-r-k+1}}{(s-r-k+2)\beta_{n-r}},$ we obtain $\beta_{s-r-k+1}'=0.$

			Thus, by induction, we conclude that, in this case, any algebra from the class $\textbf{(12)}_s (\alpha, \beta_{1},\beta_{2},\dots,\beta_{n-1})$ is isomorphic to the algebra with parameters $\beta_{n-1}'= \alpha,$ $\beta_{n-i}' = 0$ for all $2 \le i \le r-1$, and $\beta_{s-r-k+1}' = 0$ for all $1 \le k \le s-r$ that is, to $\mathbb{A}_{5,s,r}(\alpha,\beta_{s-r+1},\dots,\beta_{n-r}).$

			\item[\textbf{Case 3:}] $\beta_{n-1} = \alpha,$ $\beta_{n-i} = 0,$ $2 \le i \le {s-2},$ $\beta_{n-s+1} \neq \frac{s}{2}.$
			Choosing $A_{n-s+1} = \frac{\beta_{1}}{s - 2\beta_{n-s+1} }$ in Lemma~\ref{lem16}, we obtain $\beta_{1}'=0.$
			Then we have that in this case any algebra from the class $\textbf{(12)}_s (\alpha, \beta_{1},\beta_{2},\dots,\beta_{n-1})$ is isomorphic to the algebra $\mathbb{A}_{6,s}(\alpha,\beta_2,\dots,\beta_{n-s+1})$.

			\item[\textbf{Case 4:}] $\beta_{n-1} = \alpha,$ $ \beta_{n-i} = 0,$ $ 2 \le i \le {s-2},$ $ \beta_{n-s+1}=\frac{s}{2}.$
			Then we have that any algebra from the class $\textbf{(12)}_s (\alpha, \beta_{1},\beta_{2},\dots,\beta_{n-1})$ is isomorphic to the algebra $\mathbb{A}_{7,s}(\alpha,\beta_1,\dots,\beta_{n-s})$.
		\end{enumerate}
	\end{proof}

	\begin{theorem}\label{12-matching-algebras} Let $ \star$ be a~$ (12)$-matching structure on the associative algebra $(\mathbb{C}_n [x],\cdot)$. Then
		$\cdot$ and $\star$ are $\id$-matching or the algebra $(\mathbb{C}_n [x], \cdot, \star)$ is isomorphic to one of the following pairwise non-isomorphic algebras:
		\begin{gather*}
			\mathbb{A}_{1}(\beta_2,\dots,\beta_{n-1}); \quad \mathbb{A}_{2}(\alpha,\beta)_{\beta \neq \alpha}; \quad
			\mathbb{A}_{3,r}(\alpha), \ 2\leq r\leq n-1; \\
			\mathbb{A}_{4,s}(\alpha,\beta_s,\dots,\beta_{n-1}), \ 3\leq s\leq n-1; \\
			\mathbb{A}_{5,s,r}(\alpha,\beta_{s-r+1},\dots,\beta_{n-r}), \ 4\leq s\leq n-1, \ 2\leq r\leq s-2; \\
			\mathbb{A}_{6,s}(\alpha,\beta_2,\dots,\beta_{n-s+1}), \ \beta_{n-s+1}\neq\frac{s}{2}; \quad \mathbb{A}_{7,s}(\alpha,\beta_1,\dots,\beta_{n-s}), \ 3\leq s\leq n-1.
		\end{gather*}
	\end{theorem}

	\begin{proof}
		Let $ \star$ be a~$ (12)$-matching structure with $\cdot$ defined on the associative algebra $(\mathbb{C}_n [x],\cdot)$. Then Theorem~\ref{int} implies that it is isomorphic to one of the following pairwise non-isomorphic algebras:
		\begin{center}
			$\textbf{(12)}_1 (\beta_{1},\dots, \beta_{n-1}), \ \textbf{(12)}_2 (\alpha,\beta_{1},\dots, \beta_{n-1}), \ \textbf{(12)}_s (\alpha,\beta_{1},\dots, \beta_{n-1}), \ 3\leq s\leq n-1.$
		\end{center}
		Moreover,
		\begin{itemize}
			\item By Lemma~\ref{12-lem-alpha_1}, the algebra $\textbf{(12)}_1 (\beta_{1},\dots, \beta_{n-1})$ is isomorphic to $\mathbb{A}_{1}(\beta_2,\dots,\beta_{n-1});$
			\item By Lemma~\ref{12-lem-alpha_2}, the algebra $\textbf{(12)}_2 (\alpha,\beta_{1},\dots, \beta_{n-1})$ is isomorphic to one of the following mutually non-isomorphic algebras:
			$\mathbb{A}_{2}(\alpha,\beta)_{\beta \neq \alpha},$ or $ \mathbb{A}_{3,r}(\alpha), \quad 2\leq r\leq n-1;$

			\item By Lemma~\ref{12-lem-alpha_2 -s}, the algebra $\textbf{(12)}_s (\alpha,\beta_{1},\dots, \beta_{n-1}), \ 3\leq s\leq n-1$ is isomorphic to one of the following mutually non-isomorphic algebras:
			\begin{center}
				$\mathbb{A}_{4,s}(\alpha,\beta_s,\dots,\beta_{n-1}), \ 3\leq s\leq n-1;$ \\
				$\mathbb{A}_{5,s,r}(\alpha,\beta_{s-r+1},\dots,\beta_{n-r}), \ 4\leq s\leq n-1, \ 2\leq r\leq s-2;$\\
				$\mathbb{A}_{6,s}(\alpha,\beta_2,\dots,\beta_{n-s+1}), \quad \mathbb{A}_{7,s}(\alpha,\beta_1,\dots,\beta_{n-s}), \ 3\leq s\leq n-1.$
			\end{center}
		\end{itemize}
	\end{proof}

    %%% REFERENCES %%%
    \end{document}